\documentclass[a4paper,12pt]{amsart}
\usepackage{verbatim}
\usepackage{mathtools}
\usepackage{amsmath}
\usepackage{enumitem}
\usepackage{cite}
\usepackage{hyperref}
\setenumerate[1]{label=(\alph{*}), ref=(\alph{*})}
\setenumerate[2]{label=(\roman{*}), ref=(\roman{*})}

\newtheorem{theorem}{Theorem}[section]
\newtheorem{lemma}[theorem]{Lemma}
\newtheorem{proposition}[theorem]{Proposition}
\newtheorem{corollary}[theorem]{Corollary}
\newtheorem{ques}[theorem]{Question}

\theoremstyle{definition}

\theoremstyle{remark}
\newtheorem{remark}[theorem]{Remark}

\numberwithin{equation}{section}

\def\R{{\mathbb R}}

\newcommand{\tn}{|\mspace{-1mu}|\mspace{-1mu}|}
\DeclareMathOperator{\linspan}{span}

\DeclareMathOperator{\supp}{supp}
\DeclareMathOperator{\sign}{sign}
\newcommand{\HB}{\text{H{\kern -0.35em}B}}
\DeclareMathOperator{\co}{co}
\DeclareMathOperator{\ext}{ext}
\DeclareMathOperator{\diam}{diam}

\newcommand{\xast}{x^{\ast}}

\newcommand{\yast}{y^{\ast}}
\newcommand{\Yast}{Y^{\ast}}
\newcommand{\zast}{z^{\ast}}
\newcommand{\Zast}{Z^{\ast}}
\newcommand{\Xast}{X^{\ast}}

\newcommand{\Xastast}{X^{\ast\ast}}

\begin{document}

\title{On Daugavet indices of thickness}
\date{}

\author{Rainis Haller, Johann Langemets, Vegard Lima, Rihhard Nadel, and Abraham Rueda Zoca}
\address{Institute of Mathematics, University of Tartu, Narva mnt 18, 51009, Tartu, Estonia}
\email{rainis.haller@ut.ee, johann.langemets@ut.ee, rihhard.nadel@ut.ee}

\urladdr{\url{https://johannlangemets.wordpress.com/}}

\address[V.~Lima]{Department of Engineering Sciences, University of Agder,
Postboks 509, 4898 Grimstad, Norway.}
\email{Vegard.Lima@uia.no}

\thanks{This work was supported by the Estonian Research Council grants (PRG877) and (PSG487). R. Nadel was supported in part by the University of Tartu Foundation's CWT Estonia travel scholarship. A.~Rueda Zoca is supported by MICINN (Spain) Grant PGC2018-093794-B-I00 (MCIU, AEI, FEDER, UE), by Junta de Andaluc\'ia Grant A-FQM-484-UGR18 and by Junta de Andaluc\'ia Grant FQM-0185.}

\address[A.~Rueda Zoca]{Universidad de Granada, Facultad de Ciencias.
Departamento de An\'{a}lisis Matem\'{a}tico, 18071-Granada
(Spain)} 
\email{ abrahamrueda@ugr.es}
\urladdr{\url{https://arzenglish.wordpress.com}}



\date{}

\dedicatory{}

\begin{abstract}
Inspired by R. Whitley's thickness index the last named author recently introduced the Daugavet index of thickness of Banach
spaces. We continue the investigation of the behavior of this index and also consider two new versions of the Daugavet index of
thickness, which helps us solve an open problem which connect the
Daugavet indices with the Daugavet equation. Moreover, we will improve the formerly known estimates
of the behavior of Daugavet index on direct sums of Banach spaces by establishing sharp bounds. As a consequence of our results we prove that,
for every $0<\delta<2$, there exists a Banach space where the
infimum of the diameter of convex combinations of slices of the unit
ball is exactly $\delta$, solving an open question from the
literature.
Finally, we prove that an open question posed by Ivakhno in 2006
about the relation between the radius and diameter
of slices has a negative answer.
\end{abstract}

\maketitle

\section{Introduction}\label{sec: Introduction}

Let $X$ be a real Banach space. To measure quantitatively how far $X$
is from having the Daugavet property, the last named author introduced
in \cite{Rueda} a parameter $\mathcal{T}(X)$, called the Daugavet
index of thickness of $X$ (for the regular index of thickness see
\cite{Whitley}), where
\begin{equation*}
  \mathcal{T}(X) = \inf\left\lbrace r>0 \;\middle|\;
  \begin{tabular}{@{}l@{}}
    \text{ there exist $x\in S_X$ and a relatively weakly }\\ \text{ open $W$ in $B_X$ such that} $\emptyset\neq W \subset B(x,r)$
   \end{tabular}
  \right\rbrace.
\end{equation*}
Notice that $0\leq\mathcal T(X)\leq 2$ for any Banach space $X$ and,
for example, $\mathcal T(\ell_1)=0$, $\mathcal T(c_0)=\mathcal
T(\ell_\infty)=1$, and $\mathcal T(C[0,1])=2$ \cite[Example
4.3]{Rueda}. In fact, $\mathcal T(X)=2$ holds if and only if $X$ has
the Daugavet property \cite[Lemmata~2 and 3]{MR1784413}. Clearly,
$\mathcal T(X)=0$ for $X$ with the Radon--Nikod\'{y}m property. The
converse does not hold in general, because there exists a Banach space
$X$, where every slice of $B_X$ has diameter two (and therefore $X$
does not have the Radon--Nikod\'{y}m property), but with arbitrarily
small nonempty relatively weakly open subsets of $B_X$ (hence
$\mathcal T(X)=0$) \cite[Theorem~2.4]{MR3334951}.

Clearly, a slice of the unit ball is relatively weakly open. On the
other hand, by Bourgain's lemma \cite[Lemma~II.1]{MR912637}, every
nonempty relatively weakly open subset of the unit ball contains a
convex combination of slices. Moreover, there exists a Banach space
such that every nonempty relatively weakly open subset has diameter
two, but it also contains convex combination of slices with
arbitrarily small diameter \cite[Theorem~2.5]{MR3281132}.

The previous two examples motivate us to study further the index
$\mathcal{T}(\cdot)$ and to introduce two new related Daugavet
indices, which are in general not equal:
\begin{equation*}
  \mathcal{T}^s(X) = \inf\left\lbrace r>0 \;\middle|\;
  \begin{tabular}{@{}l@{}}
    \text{ there exist $x\in S_X$ and a slice $S$ of $B_X$}\\ \text{ such that} $S \subset B(x,r)$
   \end{tabular}
  \right\rbrace
\end{equation*}
and
\begin{equation*}
  \mathcal{T}^{cc}(X) = \inf\left\lbrace r>0 \;\middle|\;
  \begin{tabular}{@{}l@{}}
    \text{ there exist $x\in S_X$ and a convex  }\\ \text{ combination $C$ of relatively weakly open }\\ \text{ subsets of $B_X$ such that} $\emptyset\neq C \subset B(x,r)$
   \end{tabular}
  \right\rbrace.
\end{equation*}

\begin{remark}
If one replaces relatively weakly open subsets of $B_X$ in the definition of $\mathcal{T}^{cc}(X)$ with slices of $B_X$, then, by Bourgain's lemma, the index remains unchanged.
\end{remark}

Observe that 
\begin{equation*}\label{eq: indices}
0\leq \mathcal{T}^{cc}(X)\leq  \mathcal{T}(X)\leq \mathcal{T}^s(X)\leq   2.
\end{equation*}

Moreover, if in a Banach space $X$ every slice (respectively, nonempty relatively weakly open subset; convex combination of nonempty relatively weakly open subsets) of $B_X$ has diameter two, then $\mathcal{T}^s(X)\geq 1$ (respectively, $\mathcal{T}(X)\geq 1$; $\mathcal{T}^{cc}(X)\geq 1$). 

It is known that $\mathcal{T}^{cc}(X)=\mathcal{T}(X)= \mathcal{T}^s(X)=2$ if and only if $X$ has the Daugavet property. This is immediate from the following result. 
\begin{proposition}[see {\cite[Lemmata~2 and 3]{MR1784413}}]
\label{prop: equivalent Daugavet}
 Let $X$ be a Banach space. The following assertions are equivalent:
 \begin{itemize}
     \item[(i)] $X$ has the Daugavet property;
     \item[(ii)] For every $x\in S_X$, every $\varepsilon >0$, and every slice $S$ of $B_X$ there exists $y\in S$ such that $\|x-y\|\geq 2-\varepsilon$;
     \item[(iii)] For every $x\in S_X$, every $\varepsilon >0$, and every nonempty relatively weakly open subset $W$ of $B_X$ there exists $y\in W$ such that $\|x-y\|\geq 2-\varepsilon$;
     \item[(iv)] For every $\varepsilon >0$,
  every $x \in S_X$ and every convex combination $C$
  of nonempty relatively weakly open subsets of $B_X$,
  there exists $y \in C$ such that
  $\|x - y\| > 2 - \varepsilon$.
  %
\end{itemize}
\end{proposition}

Examples of Banach spaces with the Daugavet property include $\mathcal C(K,X)$ (resp. $L_1(\mu,X)$ and $L_\infty(\mu,X)$), regardless $X$, when $K$ does not have any isolated points (resp. $\mu$ does not contain any atom) \cite{wer}, the $\ell_1$-sum and the $\ell_\infty$-sum of two Banach spaces with the Daugavet property \cite{woj} or $C[0,1]\widehat{\otimes}_\pi C[0,1]$ \cite[Theorem 1.2]{rtv}.

Of course, in general one has
\[
  \mathcal{T}^s(X) \leq \inf\left\lbrace r>0 \;\middle|\;
  \begin{tabular}{@{}l@{}}
    \text{ there exist a slice $S$ of $B_X$ and $x\in S\cap S_X$}\\ \text{ such that} $S \subset B(x,r)$
   \end{tabular}
  \right\rbrace
\]
and 
\[
  \mathcal{T}(X) \leq \inf\left\lbrace r>0 \;\middle|\;
  \begin{tabular}{@{}l@{}}
    \text{ there exist a relatively weakly open $W$ in $B_X$}\\ \text{ and $x\in W\cap S_X$ such that} $W \subset B(x,r)$
   \end{tabular}
  \right\rbrace.
\]
However, notice that both of these inequalities can be strict. For example, if $X=C[0,1]\oplus_2 C[0,1]$, then both right hand side
inequalities are 2, but $\mathcal
T(X)\leq \mathcal T^s(X)<2$. This
happens for any Banach space $X$, which fails the Daugavet property, but has the diametral diameter two property (see the definition in \cite{MR3818544}).  

In Section \ref{section:directsum} we carry out a systematic study of
Daugavet indices of thickness in direct sums of Banach spaces. We
establish sharp bounds on all of the indices in $\ell_p$-sums (see
Theorem~\ref{thm: T p}), which improve the known upper estimates from
\cite{Rueda}. As an application, we prove that for each $r\in [0,2]$
there exists a Banach space $X$ such that $\mathcal T^s(X)=\mathcal
T(X)=\mathcal T^{cc}(X)=r$ (see Theorem
\ref{theorem:consindex[0,2]}). For the proof, we make use of
Proposition \ref{prop:polyhedralexample}, which allows us to solve in
Corollary \ref{corollary:polyhedral} an open question from \cite{hln}
(see the Remark after Theorem~2.8).

In Section \ref{section:remarks} we answer negatively a question posed
by the last named author in \cite[Problem~5.3]{Rueda}. Also, we will
discuss the relation of $\mathcal{T}^s(\cdot)$ between isomorphic
Banach spaces (see Proposition~\ref{T^cc closed BM-dist}), which is
then applied to prove that the Daugavet property is closed with
respect to the Banach--Mazur distance.

We end the paper by giving a negative answer to a question of Ivakhno
from 2006 \cite[p.~96]{Ivakhno}. If in a Banach space $X$ every slice
of $B_X$ has diameter two, then every slice has radius one, that is,
$X$ has the $r$-big slice property (see Section \ref{section:remarks}
for details). Ivakhno asked whether the converse is always true. We
show in Theorem~\ref{theo:JTinfty} that there exists a Banach space
$X$ such that $X$ has the $r$-big slice property, but the unit ball
contains slices of diameter smaller than $\sqrt{2}+\varepsilon$ for
every $\varepsilon>0$, which answers the above mentioned question negatively.

All Banach spaces considered in this paper are nontrivial and over
the real field. The closed unit ball of a Banach space $X$ is denoted
by $B_X$ and its unit sphere by $S_X$. The dual space of $X$ is
denoted by $X^\ast$ and the bidual by $X^{\ast\ast}$.

By a \emph{slice} of $B_X$ we mean a set of the form
\begin{equation*}
S(B_X, x^*,\alpha) :=
\{
x \in B_X : x^*(x) > 1 - \alpha
\},
\end{equation*}
where $x^* \in S_{X^*}$ and $\alpha > 0$.
A finite convex combination of slices is then of the form 
\[
\sum_{i=1}^n\lambda_i S(B_X, \xast_i, \alpha_i),
\]
where $n\in \mathbb N$ and $\lambda_i\in [0,1]$ such that $\sum_{i=1}^n\lambda_i=1$.

We recall that a norm $N$ on $\mathbb R^2$ is called \emph{absolute} (see \cite{NMR2}) if
\[
N(a,b)=N(|a|,|b|)\qquad\text{for all $(a,b)\in\mathbb R^2$}
\]
and \emph{normalized} if
\[
N(1,0)=N(0,1)=1.
\]
For example, the $\ell_p$-norm $\|\cdot\|_p$ is absolute and normalized for every $p\in[1,\infty]$. If $N$ is an absolute normalized norm on $\mathbb R^2$ (see \cite[Lemmata~21.1 and 21.2]{NMR2}), then 

\begin{itemize}
\item $\|(a,b)\|_\infty\leq N(a,b)\leq \|(a,b)\|_1$ for all $(a,b)\in\mathbb R^2$;
\item if $(a,b),(c,d)\in\mathbb R^2$ with $|a|\leq |c|\quad \text{and}\quad |b|\leq |d|,$ then \[N(a,b)\leq N(c,d);\]
\item the dual norm $N^\ast$ on $\mathbb \R^2$ defined by 
\[
N^\ast(c,d)=\max_{N(a,b)\leq 1}(|ac|+|bd|) \qquad\text{for all $(c,d)\in\mathbb R^2$}
\]
is also absolute and normalized. Note that $(N^\ast)^\ast=N$.
\end{itemize}

If $X$ and $Y$ are Banach spaces and $N$ is an absolute normalized norm on $\R^2$, then we denote by $X\oplus_N Y$ the product space $X\times Y$ with respect to the norm
\[
\|(x,y)\|_N=N(\|x\|,\|y\|) \qquad\text{for all $x\in X$ and $y\in Y$}.
\]
In the special case where $N$ is the $\ell_p$-norm, we write $X\oplus_p Y$.
Note that $(X\oplus_N Y)^\ast=\Xast\oplus_{N^\ast} \Yast$.

\section{Daugavet indices in direct sums}\label{section:directsum}

We start by recalling a result for the index $\mathcal{T}(\cdot)$ from \cite{Rueda}.
\begin{proposition}[see {\cite[Proposition~4.5]{Rueda}}]\label{prop: estimates of T}
 Let $X$ and $Y$ be Banach spaces. Then
 \begin{enumerate}
     \item\label{it: a} $\mathcal{T}(X\oplus_1 Y)\leq \min\{\mathcal{T}(X), \mathcal{T}(Y)\}$;
     \item\label{it: b} $\mathcal{T}(X\oplus_p Y)\leq (\frac{(2^{1/p}+1)^p}{2})^{1/p}$ for every $1<p<\infty$;
     \item $\mathcal T(X\oplus_\infty Y)\geq
       \min\{\mathcal T(X),\mathcal T(Y)\}$,
       where equality holds if $\mathcal T(X\oplus_\infty Y)>1$.
 \end{enumerate}
\end{proposition}

Now we provide a lower estimate for $\mathcal{T}(X \oplus_p Y)$, where $1< p<\infty$.

\begin{proposition}\label{LowerP}
  Let $X$ and $Y$ be Banach spaces, $N$ be an absolute normalized norm on $\R^2$, and $\gamma>0$ is such that $N(\cdot) \geq \gamma \|\cdot \|_1$. Then
  \begin{enumerate}
  \item 
$\mathcal{T}^s(X\oplus_N Y) \ge 2\gamma\big(\min\{\mathcal{T}^s(X), \mathcal{T}^s(Y)\}-1\big)$;
  \item 
$\mathcal{T}(X\oplus_N Y) \ge 2\gamma\big(\min\{\mathcal{T}(X), \mathcal{T}(Y)\}-1\big)$;
  \item
$\mathcal{T}^{cc}(X\oplus_N Y) \ge 2\gamma\big(\min\{\mathcal{T}^{cc}(X), \mathcal{T}^{cc}(Y)\}-1\big)$.
  \end{enumerate}
  In particular, $\mathcal{T}(X \oplus_p Y) \ge 2^{1/p}\big(\min\{\mathcal{T}(X), \mathcal{T}(Y)\}-1\big)$ whenever $1<p<\infty$.
\end{proposition}
\begin{proof}
We will only prove (b), because the proofs of (a) and (c) are very similar.

(b). Without loss of generality we assume that $\min \lbrace\mathcal{T}(X), \mathcal{T}(Y)\rbrace>1$ otherwise the lower bound trivially holds. Let $\varepsilon>0$ be such that $\min \lbrace\mathcal{T}(X), \mathcal{T}(Y)\rbrace>1+\varepsilon$.  Denote by $Z := X \oplus_N Y$.
	Let $(\tilde{x},\tilde{y}) \in S_Z$ and
	let $W$ be a nonempty relatively weakly open subset of $B_Z$.
	
	Without loss of generality we may assume that 
	\[
	W=\{z\in B_Z\colon |\zast_i(z)-\zast_i(z_0)|<1,\; i\in\{1,\dots,n\}\},
	\]
	for some $\zast_i=(\xast_i,\yast_i)\in \Zast$ and $z_0=(x_0, y_0)\in S_Z$.
	
	Define now 
	\[
	U:=
	\begin{dcases}
	\{x\in B_X\colon |\xast_i(x)-\xast_i(\frac{x_0}{\|x_0\|})|<\frac{1}{2\|x_0\|},\; i\in\{1,\dots,n\}\}, &\text{ if } x_0 \neq 0,\\
	\{x\in B_X\colon |\xast_i(x)|<\frac{1}{2},\; i\in\{1,\dots,n\}\}, &\text{ if } x_0=0,
	\end{dcases}
	\]
	and
	\[
	V:=\begin{dcases}
	\{y \in B_Y\colon |\yast_i(y)-\yast_i(\frac{y_0}{\|y_0\|})|<\frac{1}{2\|y_0\|},\; i\in\{1,\dots,n\}\}, &\text{ if } y_0 \neq 0,\\
	\{y\in B_Y\colon |\yast_i(y)|<\frac{1}{2},\; i\in\{1,\dots,n\}\}, &\text{ if } y_0=0.
	\end{dcases}
	\]
	From now on we will distinguish two cases.
	
	\textit{Case 1:} Assume first that $\tilde{x}\neq 0$ and $\tilde{y}\neq 0$. Due to the definition of the Daugavet index we can find $u\in U$ and $v\in V$ such that $\|\frac{\tilde{x}}{\|\tilde{x}\|}-u\|\geq \mathcal{T}(X)-\varepsilon/2$ and $
	\|\frac{\tilde{y}}{\|\tilde{y}\|}-v\|\geq \mathcal{T}(Y)-\varepsilon/2.$

	\textbf{Claim.} If two elements $e$ and $\tilde{e}$ in the unit ball
	of a Banach space $E$ satisfy that $\|e + \tilde{e}\| \geq 1 + \alpha$
	for some $\alpha \in [0,1]$, then
	$\|\lambda e + \mu \tilde{e}\| \geq (\lambda + \mu)\alpha$
	for all $\lambda, \mu \geq 0$.
	Indeed, the claim follows from the inequalities
	\begin{equation*}
	(\lambda + \mu)(1 + \alpha)
	\le (\lambda + \mu)\|e+\tilde{e}\|
	\le \|\lambda e+\mu \tilde{e}\| + \lambda + \mu.
	\end{equation*}
	
	Since 
	\[
	\mathcal{T}(X)-\varepsilon/2>1+\varepsilon/2,
	\]
	we can apply the Claim above and get that 
	\[
	\begin{aligned}
	\|\tilde{x}-\|x_0\|u\|&\geq (\|\tilde{x}\| +\|x_0\|)(\mathcal{T}(X)-1-\varepsilon/2).
	\end{aligned}
	\]
	Similarly, one obtains $\|\tilde{y}-\|y_0\|v\|\geq (\|\tilde{y}\|
	+\|y_0\|)(\mathcal{T}(Y)-1-\varepsilon/2).$
	
	Observe that $(\|x_0\|u,\|y_0\|v)\in W$ and
	\begin{align*}
	\|(\tilde{x}, \tilde{y})-(\|x_0\|u,\|y_0\|v)\|_N&= N\Big( \|\tilde{x} - \|x_0\|u\|, \|\tilde{y} - \|y_0\|v\|\Big)\\
	&\ge N\Big((\|\tilde{x}\| +\|x_0\|)(\mathcal{T}(X)-1-\varepsilon/2),\\
	&(\|\tilde{y}\| +\|y_0\|)(\mathcal{T}(Y)-1-\varepsilon/2)\Big)
	\\
	&\geq N\Big(\|\tilde{x}\|+\|x_0\|,\|\tilde{y}\|+\|y_0\| \Big)(\min \lbrace\mathcal{T}(X), \mathcal{T}(Y)\rbrace-1-\varepsilon/2)\\
	& \geq \gamma (\|\tilde{x}\|+\|\tilde{y}\|+\|x_0\|+\|y_0\|)(\min \lbrace\mathcal{T}(X), \mathcal{T}(Y)\rbrace-1-\varepsilon/2) \\
	&\geq 2\gamma(\min \lbrace\mathcal{T}(X), \mathcal{T}(Y)\rbrace-1-\varepsilon/2).
	\end{align*}
	Hence, by the arbitrariness of $\varepsilon$, we conclude that $\mathcal{T}(X \oplus_N Y) \ge 2\gamma(\min \lbrace\mathcal{T}(X), \mathcal{T}(Y)\rbrace-1)$.
	
	\textit{Case 2:} Assume now that $\tilde{y}= 0$, hence $\|\tilde{x}\|=1$ (the case when $\tilde{x}=0$ follows similarly). Now we can find $u\in U$ such that $\|\tilde{x}-u\|\geq \mathcal{T}(X)-\varepsilon/2$ and let $v\in V\cap S_Y$. Again, note that $(\|x_0\|u,\|y_0\|v)\in W$ and
	\begin{align*}
	\|(\tilde{x}, 0)-(\|x_0\|u,\|y_0\|v)\|_N&= N\Big(\|\tilde{x} - \|x_0\|u\| , \|y_0\|\Big)\\
	&\ge N\Big((1 +\|x_0\|)(\mathcal{T}(X)-1-\varepsilon/2), \|y_0\|\Big)
	\\
	&\ge N\Big(1 +\|x_0\|, \dfrac{\|y_0\|}{(\mathcal{T}(X)-1-\varepsilon/2)}\Big)(\mathcal{T}(X)-1-\varepsilon/2)
	\\
	&\geq \gamma(1+\|x_0\|+\|y_0\|)(\mathcal{T}(X)-1-\varepsilon/2)  \\
	&\geq 2\gamma(\mathcal{T}(X)-1-\varepsilon/2),
	\end{align*}
	Again, by the arbitrariness of $\varepsilon$, we conclude that $\mathcal{T}(X \oplus_N Y) \ge 2\gamma(\min \lbrace\mathcal{T}(X), \mathcal{T}(Y)\rbrace-1)$.
\end{proof}

\begin{remark}
With almost identical proof one can generalize Proposition~\ref{LowerP} to a finite direct sum of Banach spaces equipped with an absolute norm and the estimates will remain the same.
\end{remark}

We now turn our attention to the upper estimate of these Daugavet indices of thickness.

\begin{proposition}\label{prop: upperbound}
	Let $X$ and $Y$ be Banach spaces, $N$ be an absolute normalized norm on $\R^2$, and $\Gamma>0$ is such that $N(\cdot) \leq \Gamma \|\cdot \|_\infty$. If $(1,0)$ or $(0,1)$ is an extreme point of $B_{(\R^2,N)}$, then $\mathcal{T}^s(X\oplus_N Y) \le \Gamma$. In particular, $\mathcal{T}^{s}(X \oplus_p Y) \le 2^{1/p}$ whenever $1<p<\infty$.
\end{proposition}
\begin{proof}
	Denote by $Z := X\oplus_N Y$ and let $\varepsilon > 0$,
	$x \in S_X$ and $y \in S_Y$.
	Assume that $e=(0,1)$ is an extreme point of $B_{(\R^2,N)}$ (the proof for the other case is similar). Then $e$ is actually a strongly exposed point which allows us to fix a $\delta>0$ such that, whenever $(a,b)\in B_{(\R^2,N)}$ and $b>1-\delta$, then $|a|<\varepsilon$. 
	Find $\yast \in S_{\Yast}$ with $\yast(y) = 1$.
	If $(u,v) \in S(B_Z,(0,\yast)),\delta)$,
	then $\|v\| \ge \yast(v) > 1 - \delta$.
	By our assumption $\|u\| < \varepsilon$. Therefore
	\begin{align*}
	\|(u,v) - (x,0)\|_N
	&= N\Big(\|u-x\|,\|v\|\Big)\\
	&\le N\Big(1+\|u\|,\|v\|\Big)\\
	&\le (1+\|u\|)N(1,1)\\
	&\le \Gamma(1+\varepsilon).
	\end{align*}	
	Since $\varepsilon > 0$ was arbitrary we get
	$\mathcal{T}^s(Z) \le \Gamma$.
\end{proof}

\begin{remark}
One can also generalize Proposition~\ref{prop: upperbound} to a finite direct sum of Banach spaces equipped with an absolute norm and the estimate will remain the same.
\end{remark}

Since $\mathcal T(\cdot)\leq \mathcal T^s(\cdot)$, then the obtained upper bound from Proposition~\ref{prop: upperbound} improves the previously known estimate from Proposition~\ref{prop: estimates of T}~\ref{it: b}  in \cite{Rueda}. Moreover, from Proposition~\ref{LowerP} we know that this estimate is sharp. We summarize this in the following result.

\begin{theorem}\label{thm: T p}
  Let $X$ and $Y$ be Banach spaces and $1 < p < \infty$. If $X$ and $Y$ both have the Daugavet property, then $\mathcal T^s(X\oplus_p Y)=\mathcal T(X\oplus_p Y)=\mathcal T^{cc}(X\oplus_p Y)=2^{1/p}$.
\end{theorem}
\begin{proof}
Since $X$ and $Y$ have the Daugavet property, then $\mathcal T^{cc}(X)=\mathcal T^{cc}(Y)=2$ by Proposition~\ref{prop: equivalent Daugavet}. From Proposition~\ref{LowerP} we get that $\mathcal T^{cc}(X\oplus_p Y)\geq 2^{1/p}$ and by Proposition~\ref{prop: upperbound} we have that $\mathcal T^{s}(X\oplus_p Y)\leq 2^{1/p}$, therefore
\[
2^{1/p}\leq \mathcal T^{cc}(X\oplus_p Y)\leq \mathcal T(X\oplus_p Y)\leq \mathcal T^{s}(X\oplus_p Y)\leq 2^{1/p},
\]
which completes the proof.
\end{proof}

As a consequence of Theorem~\ref{thm: T p} we get the following result.

\begin{theorem}\label{theorem:consindex[0,2]}
For every $r\in [0,2]$ there exists a Banach space $X$ such that $\mathcal{T}^s(X)=\mathcal{T}(X)=\mathcal T^{cc}(X)=r$.
\end{theorem}

For the proof we will need the following result.

\begin{proposition}\label{prop:polyhedralexample}
  For every $r>0$ there exists a Banach space $X$ such that
  \begin{equation*}
    \mathcal T^{cc}(X)=\mathcal T(X)=\mathcal T^s(X)=\frac{r}{1+r}.
  \end{equation*}
\end{proposition}

\begin{proof}
  The proof is inspired by the example exhibited in
  \cite[Theorem~2.1]{loru}.
  Pick an arbitrary $r>0$. Define
  \begin{equation*}
    U^* := \co\left(
      B_{\ell_1\oplus_\infty\mathbb R}
      \cup
      \left\{ \left(0,1+r\right), \left(0,-1-r\right) \right\}
    \right),
  \end{equation*}
  which is clearly a weak$^*$ compact set in
  $(c_0\oplus_1\mathbb{R})^*$.
  Consequently, there is a norm $\tn\cdot\tn$ on
  $c_0\oplus \mathbb{R}$ whose unit ball is
  \begin{equation*}
    U := \{ (x,\beta) \in c_0\oplus \mathbb R:
    \phi(x,\beta)\leq 1\ \mbox{ for all } \phi\in U^*\}.
  \end{equation*}
  Consider $X:=(c_0\oplus\mathbb R, \tn\cdot\tn)$, and let us prove
  that $X$ satisfies the desired requirements. It is clear that
  $U^*=B_{X^*}$. Also, it is clear that
  \begin{equation}\label{exam:polyextrem}
    \ext(U^*)
    =
    \left\{
      \left( 0,\pm\left(1+r\right) \right)
    \right\}
    \cup
    \{(\xi e_n,\psi 1) :
    n\in\mathbb{N} \mbox{ and } \xi,\psi\in\{-1,1\}
    \}.
  \end{equation}
  Therefore, $\ext(U^*)'=\{(0,\pm 1)\}\subseteq \frac{1}{1+r}U^*$.

  Note also that $B_{\ell_1\oplus_\infty\mathbb R} \subseteq
  U^*\subseteq (1+r)B_{\ell_1\oplus_\infty\mathbb R}$, so
  \begin{equation*}
    \frac{1}{1+r}B_{c_0\oplus_1\mathbb R}
    \subseteq U
    \subseteq B_{c_0\oplus_1\mathbb R}.
  \end{equation*}
  Consequently, for each pair $(x,\beta)\in X$, it follows that
  \begin{equation}\label{exam:desiteonorma}
    \Vert (x,\beta)\Vert_1
    \leq
    \tn(x,\beta)\tn
    \leq (1+r)\Vert (x,\beta)\Vert_1.
  \end{equation}
  Define $S_\delta:=S(B_X,(0,1+r),\delta)$.
  Pick an element $(x,\beta)\in S_\delta$,
  and let us estimate $\tn(x,\beta)-(0,\frac{1}{1+r})\tn$.
  To this end, notice that
  \begin{equation*}
    1 \geq (1+r) \beta = (0,1+r)(x,\beta) > 1-\delta
    \Rightarrow
    \frac{1}{1+r} \geq \beta \geq \frac{1-\delta}{1+r}.
  \end{equation*}
  We claim that $\Vert x\Vert_\infty\leq \frac{r+\delta}{1+r}$.
  Assume for contradiction that there exists
  $n\in\mathbb{N}$
  such that $\vert x(n)\vert > \frac{r+\delta}{1+r}$,
  choose $\xi := \sign(x(n))$ and define
  $x^* := (\xi e_n,1) \in U^*\subseteq B_{X^*}$.
  Then
  \begin{equation*}
    1 \geq x^*((x,\beta))
    = \vert x(n) \vert + \beta
    > \frac{r+\delta}{1+r} + \frac{1-\delta}{1+r}
    = \frac{r+\delta+1-\delta}{1+r}
    =1,
  \end{equation*}
  a contradiction.
  So $\Vert x\Vert_\infty\leq \frac{r+\delta}{1+r}$.
  Also, notice that $(0,\frac{1}{1+r})\in S_{\delta}$
  (note that
  $\Vert (0,\frac{1}{1+r})\tn \leq
  (1+r)\Vert (0,\frac{1}{1+r})\Vert_1 = 1$
  by \eqref{exam:desiteonorma}).

  Let us estimate $\Vert (x,\beta)-(0,\frac{1}{1+r})\Vert$.
  By the Krein--Milman theorem
  \begin{align*}
    \Vert (x,\beta) - (0,\frac{1}{1+r}) \Vert
    =
    \sup\limits_{x^*\in \ext(U^*)}
    \vert x^*((x,\beta)-(0,\frac{1}{1+r}))\vert.
  \end{align*}
  Given $x^*\in \ext(U^*)$,
  then $x^*=(y^*,\lambda)$ for $y^*\in \ell_1$
  and $\lambda \in \mathbb{R}$.
  From now on we will distinguish two cases.

	\textit{Case 1:} Assume first that  $y^*=0$.
    This implies, according to \eqref{exam:polyextrem},
    that $\vert \lambda\vert=1+r$.
    Since $\frac{1}{1+r}\geq \beta\geq \frac{1-\delta}{1+r}$
    we get that
    \begin{equation*}
      \vert x^*((x,\beta) - (0,\frac{1}{1+r})) \vert
      = \vert \lambda \vert \vert \beta-\frac{1}{1+r} \vert
      \leq (1+r)\frac{\delta}{1+r}
      = \delta.
    \end{equation*}

	\textit{Case 2:} Assume now that $y^*\neq 0$, then, by \eqref{exam:polyextrem},
    $\vert \lambda\vert=1$ and $y^* = \pm e_k$
    for suitable $k \in \mathbb{N}$.
    Hence
    \begin{align*}
      \vert x^*((x,\beta) - (0,\frac{1}{1+r})) \vert
      &=
      \vert y^*(x) + \lambda(\beta-\frac{1}{1+r}) \vert
      \\
      &\leq \Vert x\Vert_\infty + \vert \lambda \vert
      \vert \beta - \frac{1}{1+r} \vert
      \\
      &\leq \frac{r+\delta}{1+r}+\frac{\delta}{1+r}.
    \end{align*}

  Taking into account the above inequalities we get that
  \begin{equation*}
    \Vert (x,\beta) - (0,\frac{1}{1+r})\Vert
    \leq \frac{r+\delta}{1+r} + \frac{\delta}{1+r}.
  \end{equation*}
  This means that
  $S_\delta \subseteq B((0,\frac{1}{1+r}),
  \frac{r+\delta}{1+r} + \frac{\delta}{1+r})$,
  so $\mathcal{T}^s(X) \leq
  \frac{r+\delta}{1+r} + \frac{\delta}{1+r}$.
  Since $\delta>0$ was arbitrary we get that
  $\mathcal T^s(X)\leq \frac{r}{1+r}$.
  
  For the second part of the proof,
  let us prove that $\mathcal T^{cc}(X)\geq \frac{r}{1+r}$,
  for which we will prove that every convex combination of slices of
  $B_X$ has diameter at least $\frac{2r}{1+r}$.
  The proof will be motivated by \cite[Proposition~2.2]{loru}.
  Take a convex combination of slices
  $C := \sum_{i=1}^n \lambda_i S(B_X, x_i^*,\alpha_i)$,
  a point $\sum_{i=1}^n \lambda_i x_i\in C$, and
  $0 < \varepsilon < \frac{r}{1+r}$.
  Given $i \in \{1,\ldots, n\}$ define
  \begin{equation*}
    A_i :=
    \{
    f\in \ext(B_{X^*}) :
    \vert f(x_i) \vert > \frac{1}{1+r} + \varepsilon
    \}.
  \end{equation*}
  Since $\ext(B_{X^*})' \subseteq \frac{1}{1+r}B_{X^*}$,
  a compactness argument implies that $A_i$ is finite.
  Consequently, we can take
  \begin{equation*}
    y \in \left(
      \bigcap\limits_{i=1}^n
      \bigcap\limits_{f\in A_i}
      \ker(f)\cap \ker(x_i^*)
    \right)
    \cap S_X.
  \end{equation*}
  Pick $i \in \{1,\ldots, n\}$. We claim that
  \begin{equation*}
    x_i \pm
    ( \frac{r}{1+r}-\varepsilon ) y
    \in S(B_{X^*}, x_i^*, \alpha_i).
  \end{equation*}
  First, notice that
  \begin{equation*}
    x_i^*(x_i\pm ( \frac{r}{1+r}-\varepsilon )y)
    =
    x_i^*(x_i)
    >
    1-\alpha
  \end{equation*}
  since $y\in \ker(x_i^*)$.
  On the other hand let us prove that
  $\Vert x_i\pm ( \frac{r}{1+r}-\varepsilon )y\Vert\leq 1$.
  To this end, notice that
  \begin{equation*}
    \Vert x_i \pm ( \frac{r}{1+r} - \varepsilon )y \Vert
    = \sup\limits_{f\in \ext(B_{X^*})}
    \vert f(x_i \pm (\frac{r}{1+r} - \varepsilon)y )\vert.
  \end{equation*}
  Given $f\in \ext(B_{X^*})$ we have two cases to consider:

	\textit{Case 1:}    If $f\in A_i$ we get that $f(y)=0$, and so
    \[
    \vert f(x_i\pm ( \frac{r}{1+r}-\varepsilon )y) \vert
    = \vert f(x_i)\vert \leq 1.
    \]

	\textit{Case 2:}   If $f\notin A_i$, then
    $\vert f(x_i)\vert \le \frac{1}{1+r} + \varepsilon$, and so
    \begin{equation*}
      \vert f(x_i\pm ( \frac{r}{1+r}-\varepsilon )y) \vert
      \leq
      \vert f(x_i)\vert + ( \frac{r}{1+r}-\varepsilon )\vert f(y)\vert
      \leq \frac{1}{1+r}+\frac{r}{1+r}=1.
    \end{equation*}

  This implies that
  $\sum_{i=1}^n\lambda_i  (x_i\pm ( \frac{r}{1+r}-\varepsilon )y)\in C$,
  so
  \begin{align*}
    \diam(C)
    &\geq
    \left\Vert
      \sum_{i=1}^n \lambda_i
      \left(x_i+( \frac{r}{1+r} - \varepsilon )y\right)
      -
      \sum_{i=1}^n \lambda_i
      \left( x_i-( \frac{r}{1+r} - \varepsilon )y\right)
    \right\Vert\\
    &= 2 \frac{r}{1+r} - 2\varepsilon.
  \end{align*}
  Since $\varepsilon>0$ was arbitrary we deduce that
  $\diam(C)\geq \frac{2r}{1+r}$.
  Consequently, if there exists a convex combination of slices $C$
  such that $C\subseteq B(0,\rho)$, then
  \begin{equation*}
    \frac{2r}{1+r} \leq \diam(C)
    \leq \diam(B(0,\rho))
    = 2\rho\Rightarrow \rho \geq \frac{r}{1+r}.
  \end{equation*}
  This implies that $\mathcal T^{cc}(X)\geq \frac{r}{1+r}$.

  Hence we have proved that
  \begin{equation*}
    \frac{r}{1+r}
    \leq \mathcal{T}^{cc}(X)
    \leq \mathcal{T}(X)
    \leq \mathcal{T}^s(X)
    \leq \frac{r}{1+r},
  \end{equation*}
  as desired.
\end{proof}

\begin{proof}[Proof of Theorem \ref{theorem:consindex[0,2]}]
First observe that, if $r$ equals $0,1$ or $2$, then take $X$ to be $\ell_1$, $c_0$ or $C[0,1]$, respectively. If $r\in (0,1)$, then there exists a $s\in (0,\infty)$ such that $r=\frac{s}{s+1}$ and apply Proposition~\ref{prop:polyhedralexample} to $s$. If $r\in(1,2)$, then there exists a $p\in(1,\infty)$ such that $r = 2^{1/p}$ and apply Theorem~\ref{thm: T p} to $X=C[0,1]\oplus_p C[0,1]$.
\end{proof}

As a consequence of our Proposition~\ref{prop:polyhedralexample} we obtain the following result.

\begin{corollary}\label{corollary:polyhedral}
For every $\delta\in (0,2)$ there exists a Banach space $X$ such that every convex combination of slices has diameter $\geq \delta$ but such that, for every $\varepsilon>0$, there exists a slice $S$ of $B_X$ such that $\diam(S)\leq \delta+\varepsilon$.
\end{corollary}

The previous corollary extends \cite[Corollary 2.9]{hln}, where it was proved for the cases $\delta\in (1,2)$, and answers an open question (see the Remark after Theorem 2.8 in \cite{hln}).

In Proposition~\ref{prop: estimates of T}~\ref{it: a}, it seems to be
unknown whether the inequality can be strict
(see Question~\ref{ques: 1-sum}).
However, for the index $\mathcal{T}^s(\cdot)$,
we always have equality.

\begin{proposition}\label{prop: estimates of Ts}
Let $X$ and $Y$ be Banach spaces. Then 
 \begin{enumerate}
     \item $\mathcal{T}^s(X\oplus_1 Y)= \min\{\mathcal{T}^s(X), \mathcal{T}^s(Y)\}$;
     \item\label{it: Ts b} $\mathcal{T}^s(X\oplus_p Y)\leq 2^{1/p}$ for every $1<p<\infty$;
     \item $\mathcal T^s(X\oplus_\infty Y)\geq
       \min\{\mathcal T^s(X),\mathcal T^s(Y)\}$,
       where equality holds if $\mathcal T^s(X\oplus_\infty Y)>1$.
 \end{enumerate}
\end{proposition}
\begin{proof}
(a). Let us first show that $\mathcal{T}^s(X\oplus_1 Y)\geq
\min\{\mathcal{T}^s(X), \mathcal{T}^s(Y)\}$. Set $Z:= X\oplus_1 Y$ and
let $S(B_Z, (\xast, \yast), \alpha)$ be a slice of $B_{Z}$, $(x,y)\in
S_Z$, and $\varepsilon>0$.

Without loss of generality suppose that $\|\xast\|=1$. Thus we have two cases either $x=0$ or $x\neq 0$. 

\textit{Case 1:} Assume first that $x=0$, hence $\|y\|=1$. Now find an element $u\in S_X$ such that $\xast(u)>1-\alpha$. Observe that $(u,0)\in S(B_Z, (\xast, \yast), \alpha)$ and
\[
\|(u,0)-(0,y)\|=2\geq \mathcal{T}^s(X).
\]

\textit{Case 2:} Assume now that $x\neq 0$. Consider the slice $S(B_X,\xast, \alpha)$ and $\frac{x}{\|x\|}\in S_X$. Find an $u\in S(B_X,\xast, \alpha)$ such that $\|\frac{x}{\|x\|}-u\|\geq \mathcal{T}^s(X)-\varepsilon$. Now $(u,0)\in S(B_Z,(\xast, \yast), \alpha)$ and 
\begin{align*}
    \|(x,y)-(u,0)\|&=\|x-u\|+\|y\|\\
    &\geq \|\frac{x}{\|x\|}-u\|-\|\frac{x}{\|x\|}-x\|+\|y\|\\
    &\geq \mathcal{T}^s(X)-\varepsilon-(1-\|x\|)+\|y\|\\
    &=\mathcal{T}^s(X)-\varepsilon.
\end{align*}
Therefore, $\mathcal{T}^s(X\oplus_1 Y)\geq \min\{\mathcal{T}^s(X), \mathcal{T}^s(Y)\}$.

The proof of $\mathcal{T}^s(X\oplus_1 Y)\leq \min\{\mathcal{T}^s(X),\mathcal{T}^s(Y)\}$ is the same as the proof of \cite[Proposition~4.5 (2)]{Rueda}, except with $m=1$.

(b). This follows immediately from  Proposition~\ref{prop: upperbound}.

 (c). Let us first show that $\mathcal{T}^s(X\oplus_\infty Y)\geq \min\{\mathcal{T}^s(X), \mathcal{T}^s(Y)\}$. Denote by $Z:=X\oplus_\infty Y$. Let $S(B_{Z},(x^{*},y^{*}),\alpha)$ be a slice of $B_Z$, $(x,y)\in  S_{Z}$, and $\varepsilon>0$. 

Define
	\begin{displaymath}
	S^X:=
	\begin{dcases}
	S(B_X,\frac{\xast}{\|\xast\|}, \alpha), &\text{ if } x^{*}\neq 0,\\
	B_X, &\text{ if } x^{*}=0,
	\end{dcases}
		\end{displaymath}
and 
		\begin{displaymath}
	S^Y:=
	\begin{dcases}
	S(B_Y,\frac{\yast}{\|\yast\|},\alpha), &\text{ if } y^{*}\neq 0,\\
	B_Y, &\text{ if } y^{*}=0.
	\end{dcases}
	\end{displaymath}

	Observe that $S^X\times S^Y \subset S(B_{Z},(x^{*},y^{*}),\alpha)$. Since $\max\{\|x\|, \|y\|\}=1$, we will suppose from now on that $\|x\|=1$. Hence there exists an $x_0 \in S^X$ such that $\|x_0-x\| > \mathcal{T}^s(X) - \varepsilon$. Let $y_0\in S^Y$ be arbitrary. We have that
	
	\begin{displaymath}
	\begin{aligned}
	\|(x_0,y_0)-(x,y)\| &= \max \lbrace \|x_0-x\|, \|y_0-y\| \rbrace \geq  \mathcal{T}^s(X) - \varepsilon.
	\end{aligned}
	\end{displaymath}
	The case when $\|y\|=1$ is similar. Therefore, by the arbitrariness of $\varepsilon$, we see that $\mathcal{T}^s(X\oplus_\infty Y)\geq \min\{\mathcal{T}^s(X), \mathcal{T}^s(Y)\}$.

	Assume now that $\mathcal{T}^s(X\oplus_\infty Y)>1$ and let us show that then $\mathcal{T}^s(X\oplus_\infty Y) \leq \min \lbrace \mathcal{T}^s(X),\mathcal{T}^s(Y) \rbrace$. Pick an $\varepsilon>0$ such that $\mathcal{T}^s(X\oplus_\infty Y)-\varepsilon>1$. Let $x\in S_X$ and $S(B_X,\xast, \alpha)$ be a slice of $B_X$. Observe that $S(B_Z,(\xast, 0), \alpha)$ is a slice of $B_{Z}$ and $(x,0)\in S_{Z}$. Thus there is an element $(u,v)\in S(B_Z,(\xast,0),\alpha)$ such that 
	\[
	1<\mathcal{T}^s(X\oplus_\infty Y)-\varepsilon\leq \|(x,0)-(u,v)\|=\max\{\|x-u\|, \|v\|\}.
	\]
	Since $\|v\|\leq 1$, then we must have that $\|x-u\|\geq \mathcal{T}^s(X\oplus_\infty Y)-\varepsilon$. Notice that $u\in S(B_X,\xast,\alpha)$, hence $\mathcal{T}^s(X)\geq \mathcal{T}^s(X\oplus_\infty Y)-\varepsilon$. Finally, by the arbitrariness of $\varepsilon$, we conclude that $\mathcal{T}^s(X\oplus_\infty Y) \leq \min \lbrace \mathcal{T}^s(X),\mathcal{T}^s(Y) \rbrace$. 
\end{proof}

\begin{remark}\label{rem: strict inequality Ts}
The inequality in Proposition~\ref{prop: estimates of Ts} (c) can be strict if we remove the assumption on $\mathcal T^s(X\oplus_\infty Y)$. Indeed, let $X=c_0$ and $Y=\R$, then $X\oplus_\infty Y$ is isometrically isomorphic to $c_0$ and 
\[
T^s(X\oplus_\infty Y)=1>0=\mathcal{T}^s(\R)=\min \lbrace \mathcal{T}^s(X),\mathcal{T}^s(Y) \rbrace.
\]
\end{remark}

We end this section by studying the index $\mathcal{T}^{cc}(\cdot)$ in $\ell_p$-sums.

\begin{proposition}\label{prop: estimates of Tcc}
Let $X$ and $Y$ be Banach spaces. Then 
 \begin{enumerate}
     \item $\mathcal{T}^{cc}(X\oplus_1 Y)\leq \min\{\mathcal{T}^{cc}(X), \mathcal{T}^{cc}(Y)\}$;
     \item\label{it: Tcc b} $\mathcal{T}^{cc}(X\oplus_p Y)\leq 2^{1/p}$ for every $1<p<\infty$;
     \item $\mathcal T^{cc}(X\oplus_\infty Y)\geq
       \min\{\mathcal T^{cc}(X),\mathcal T^{cc}(Y)\}$,
       where equality holds if $\mathcal T^{cc}(X\oplus_\infty Y)>1$.
 \end{enumerate}
\end{proposition}
\begin{proof}
(a). Let $\varepsilon>0$, and assume without loss of generality that
$\min\{\mathcal T^{cc}(X),\mathcal T^{cc}(Y)\}=\mathcal
T^{cc}(X)$. Find a convex combination of slices $\sum_{i=1}^n
\lambda_i S(B_X, \xast_i, \alpha)$ of $B_X$ and an $x\in S_X$ such
that
\[
\sum_{i=1}^n \lambda_i S(B_X, \xast_i, \alpha)\subset B(x, \mathcal{T}^{cc}(X)+\varepsilon).
\]

Let $\delta\in (0,\alpha)$ and set $Z:=X\oplus_1 Y$. Observe that 
\[
S(B_{Z}, (\xast_i,0), \delta)\subset S(B_X, \xast_i, \alpha)\times \delta B_Y
\] 
for every $i\in\{1,\dots,n\}$. Therefore,
\begin{align*}
    \sum_{i=1}^n \lambda_i S(B_{Z}, (\xast_i,0), \delta)&\subset\sum_{i=1}^n \lambda_i S(B_X, \xast_i, \alpha)\times\delta B_Y\\
    &\subset B(x, \mathcal{T}^{cc}(X)+\varepsilon)\times \delta B_Y\\
    &\subset B((x,0), \mathcal{T}^{cc}(X)+\varepsilon+\delta).
\end{align*}
Since $\varepsilon$ and  $\delta$ can be chosen to be arbitrarily small, we have that $\mathcal{T}^{cc}(X\oplus_1 Y)\leq \min\{\mathcal{T}^{cc}(X), \mathcal{T}^{cc}(Y)\}$.

(b). This follows immediately from the inequality $\mathcal{T}^{cc}(\cdot)\leq \mathcal{T}^{s}(\cdot)$ and Proposition~\ref{prop: estimates of Ts}~\ref{it: Ts b}.

(c). Let us first show that $\mathcal{T}^{cc}(X\oplus_\infty Y)\geq \min\{\mathcal{T}^{cc}(X), \mathcal{T}^{cc}(Y)\}$. Denote by $Z:=X\oplus_\infty Y$. Let $n\in \mathbb N$, for every $i\in\{1,\dots,n\}$ let $S(B_Z,(x^{*}_{i},y^{*}_{i}),\alpha)$ be slices of $B_Z$, $\lambda_{i} > 0$ with $\sum_{i=1}^{n}\lambda_{i}=1$, $(x,y)\in  S_{Z}$, and $\varepsilon>0$. Denote by $S:=\sum_{i=1}^{n}\lambda_{i}S(B_Z,(x^{*}_{i},y^{*}_{i}),\alpha)$.
	
	Define
	\begin{displaymath}
	S^X_{i}:=
	\begin{dcases}
	S(B_X,\frac{\xast_{i}}{\|\xast_{i}\|}, \alpha), &\text{ if } \xast_{i}\neq 0,\\
	B_X, &\text{ if } \xast_{i}=0,
	\end{dcases}
	\end{displaymath}
	and 
	\begin{displaymath}
	S^Y_{i}:=
	\begin{dcases}
	S(B_Y,\frac{\yast_{i}}{\|\yast_{i}\|},\alpha), &\text{ if } \yast_{i}\neq 0,\\
	B_Y, &\text{ if } \yast_{i}=0.
	\end{dcases}
	\end{displaymath}
	
	 Denote by $S^{X}:=\sum_{i=1}^{n}\lambda_{i}S^{X}_{i}$ and $S^{Y}:=\sum_{i=1}^{n}\lambda_{i}S^{Y}_{i}$. Notice that $S^{X}_{i}\times S^{Y}_{i} \subset S(B_Z,(x^{*}_{i},y^{*}_{i}),\alpha)$ and that therefore $S^{X}\times S^{Y}\subset S$. Since $\max\{\|x\|, \|y\|\}=1$, we will suppose from now on that $\|x\|=1$. Hence there exists an $x_0\in S^{X}$ such that $\|x_0-x\| > \mathcal{T}^{cc}(X) - \varepsilon$. Let $y_0\in S^Y$ be arbitrary. We have that
	
	\begin{displaymath}
	\begin{aligned}
	\|(x_0,y_0)-(x,y)\| &= \max \lbrace \|x_0-x\|, \|y_0-y\| \rbrace \geq  \mathcal{T}^{cc}(X) - \varepsilon.
	\end{aligned}
	\end{displaymath}
	The case when $\|y\|=1$ is similar. Therefore, by the arbitrariness of $\varepsilon$, we see that $\mathcal{T}^{cc}(X\oplus_\infty Y)\geq \min\{\mathcal{T}^{cc}(X), \mathcal{T}^{cc}(Y)\}$.

	Assume now that $\mathcal{T}^{cc}(X\oplus_\infty Y)>1$ and let us show that then $\mathcal{T}^{cc}(X\oplus_\infty Y) \leq \min \lbrace \mathcal{T}^{cc}(X),\mathcal{T}^{cc}(Y) \rbrace$. Pick an $\varepsilon>0$ such that $\mathcal{T}^{cc}(X\oplus_\infty Y)-\varepsilon>1$. Let $x\in S_X$, $S(B_X,\xast_{i}, \alpha)$ be slices of $B_X$, and $\lambda_{i} > 0$, such that $\sum_{i=1}^{n}\lambda_{i}=1$. Observe that for each $i$ we have that $S(B_Z,(\xast_{i}, 0), \alpha)$ is a slice of $B_{Z}$ and $(x,0)\in S_{Z}$. Thus there is an element $(u,v)\in \sum_{i=1}^{n}\lambda_{i}S(B_Z,(\xast_{i},0),\alpha)$ such that 
	\[
	1<\mathcal{T}^{cc}(X\oplus_\infty Y)-\varepsilon\leq \|(x,0)-(u,v)\|=\max\{\|x-u\|, \|v\|\}.
	\]
	Since $\|v\|\leq 1$, then we must have that $\|x-u\|\geq \mathcal{T}^{cc}(X\oplus_\infty Y)-\varepsilon$. Notice that $u\in \sum_{i=1}^{n}\lambda_{i}S(B_X,\xast_{i},\alpha)$, hence $\mathcal{T}^{cc}(X)\geq \mathcal{T}^{cc}(X\oplus_\infty Y)-\varepsilon$. Finally, from the arbitrariness of $\varepsilon$, we conclude that $\mathcal{T}^{cc}(X\oplus_\infty Y) \leq \min \lbrace \mathcal{T}^{cc}(X),\mathcal{T}^{cc}(Y) \rbrace$.
\end{proof}

\begin{remark}
The same example as in Remark~\ref{rem: strict inequality Ts} shows that the inequality in Proposition~\ref{prop: estimates of Tcc} (c) can be strict if we remove the assumption on $\mathcal T^{cc}(X\oplus_\infty Y)$.
\end{remark}

Recall that from Proposition~\ref{prop: estimates of Ts} (a) we know that $\mathcal{T}^{s}(X\oplus_1 Y)= \min\{\mathcal{T}^{s}(X), \mathcal{T}^{s}(Y)\}$ holds for all Banach spaces $X$ and $Y$. However, we do not know whether the corresponding equalities hold for the indices $\mathcal{T}(\cdot)$ and $\mathcal{T}^{cc}(\cdot)$ too.

\begin{ques}\label{ques: 1-sum}
Let $X$ and $Y$ be Banach spaces.
\begin{enumerate}
    \item $\mathcal{T}(X\oplus_1 Y)= \min\{\mathcal{T}(X), \mathcal{T}(Y)\}$?
    \item $\mathcal{T}^{cc}(X\oplus_1 Y)= \min\{\mathcal{T}^{cc}(X), \mathcal{T}^{cc}(Y)\}$?
\end{enumerate}

\end{ques}

\section{Remarks and open questions}\label{section:remarks}

In a dual Banach space one can also consider the weak$^\ast$ versions of the Daugavet indices of thickness. In \cite{Rueda} the weak$^\ast$ version of $\mathcal{T}(\cdot)$, denoted by $\mathcal{T}_{w^{\ast}}(\cdot)$, was introduced. For a Banach space $X$ we will also consider
\begin{equation*}
  \mathcal{T}^s_{w^{\ast}}(\Xast) = \inf\left\lbrace r>0 \;\middle|\;
  \begin{tabular}{@{}l@{}}
    \text{ there exist $\xast\in S_{\Xast}$ and a weak$^*$ slice }\\ \text{  $S$ of $B_{\Xast}$ such that} $S \subset B(\xast,r)$
   \end{tabular}
  \right\rbrace
\end{equation*}
and

\begin{equation*}
  \mathcal{T}_{w^{\ast}}^{cc}(\Xast) = \inf\left\lbrace r>0 \;\middle|\;
  \begin{tabular}{@{}l@{}}
     \text{ there exist $\xast\in S_{\Xast}$ and a convex  }\\ \text{ combination $C$ of relatively weak$^\ast$ open }\\ \text{ subsets of $B_{\Xast}$ such that} $\emptyset\neq C \subset B(\xast,r)$
   \end{tabular}
  \right\rbrace.
\end{equation*}

Clearly, for any Banach space $X$ we have that 
\begin{equation}\label{eq: w*-indices}
0\leq \mathcal{T}_{w^{\ast}}^{cc}(\Xast)\leq \mathcal{T}_{w^{\ast}}(\Xast)\leq \mathcal{T}_{w^{\ast}}^s(\Xast)\leq 2,
\end{equation}

and observe that
\begin{equation}\label{eq: w*-indices2}
   \mathcal{T}^s(\Xastast)\leq  \mathcal{T}_{w^{\ast}}^s(\Xastast)\leq \mathcal{T}^s(X)
\end{equation}
and
\begin{equation}\label{eq: w*-indices3}
    \mathcal{T}(\Xastast) \leq \mathcal{T}_{w^{\ast}}(\Xastast)\leq \mathcal{T}(X),
\end{equation}
and
\begin{equation}\label{eq: w*-indices4}
    \mathcal{T}^{cc}(\Xastast)\leq \mathcal{T}^{cc}_{w^{\ast}}(\Xastast)\leq \mathcal{T}^{cc}(X).
\end{equation}

\begin{remark}\label{remark:relaindiweakstar}
Let us make some observations on the above indices:
\begin{enumerate}
\item By considering the biduals of the Banach spaces which give us the strict inequalities between the regular indices and taking into account (\ref{eq: w*-indices2})--(\ref{eq: w*-indices4}) one has that the inequalities in (\ref{eq: w*-indices}) can in general be strict.
\item Given a dual Banach space $X^*$, the inequality $\mathcal T(X^*)\leq \mathcal T_{w^*}(X)$ may be strict. Indeed, if $C[0,1]$, then $\mathcal T^{cc}_{w^*}(X^*)=2$ since $X$ has the Daugavet property. However, $\mathcal T^s(X^*)=0$ since $B_{X^*}$ contains slices of arbitrarily small diameter. This shows that the first inequality of (\ref{eq: w*-indices2})--(\ref{eq: w*-indices4}) can be strict.
\item Again take $X=C[0,1]$. It satisfies that $\mathcal T^{cc}(X)=2$ since $X$ has the Daugavet property. However, $\mathcal T^{s}_{w^*}(X^{**})<2$ since $X^*$ fails the Daugavet property. This shows that the second inequality of (\ref{eq: w*-indices2})--(\ref{eq: w*-indices4}) can be strict.
\end{enumerate}
\end{remark}

In \cite[Problem~5.3]{Rueda} it is wondered whether the equality 
\begin{equation}\label{eq: weakly compact}
\inf\left\lbrace \|T+I\| \;\middle|\;
  \begin{tabular}{@{}l@{}}
    \text{ $T\in \mathcal L(X)$, $\|T\|=1$, and}\\ \text{ $T$ is weakly compact} 
   \end{tabular}
  \right\rbrace
  =\max\{\mathcal T(X), \mathcal{T}_{w^\ast}(\Xast)\}
\end{equation}
holds for every Banach space $X$. We will now show that equality~(\ref{eq: weakly compact}) does not hold in general. We begin by observing that the proof of \cite[Proposition~4.4]{Rueda} actually shows that.

\begin{proposition}\label{prop: weakly compact estimate}
  Let $X$ be a Banach space. Then, for every norm one and weakly compact operator $T\colon X\to X$, it follows that 
  \[
  \|T+I\|\geq \max\{\mathcal{T}^s(X), \mathcal{T}^s_{w^\ast}(X^\ast)\}.
  \]
\end{proposition}

By \cite[Theorem~2.4]{MR3334951}, there exists
an equivalent renorming $Z$ of $c_0$ such that all slices
of $B_Z$ have diameter two and there are
relatively weakly open subsets of $B_Z$ with arbitrarily
small diameter. Then $\mathcal{T}^s(Z)\geq 1$, but
$\mathcal{T}(Z)=0=\mathcal{T}_{w^{\ast}}(Z^\ast)$ (notice that $Z^*$
has the Radon--Nikod\'{y}m property because it is isomorphic to
$\ell_1$, and the result follows from \cite[Theorem~11.8]{fab}).
Therefore, the equality~(\ref{eq: weakly compact})
fails for this Banach space $Z$.

\begin{ques}
Does the equality 
\begin{equation*}
\inf\left\lbrace \|T+I\| \;\middle|\;
  \begin{tabular}{@{}l@{}}
    \text{ $T\in \mathcal L(X)$, $\|T\|=1$,}\\ \text{ and $T$ is weakly compact} 
   \end{tabular}
  \right\rbrace
  =\max\{\mathcal T^s(X), \mathcal{T}^s_{w^\ast}(\Xast)\}
\end{equation*}
hold for every Banach space $X$?
\end{ques}

Our next aim is to show that the Daugavet index $\mathcal{T}^s(\cdot)$ behaves well with respect to the Banach--Mazur distance. Recall that this distance between two isomorphic Banach spaces $X$ and $Y$ is defined by
\[
d(X,Y):=\inf \{\|L\|\|L^{-1}\|\colon L\colon X\to Y \text{ is an isomorphism}\}.
\]

\begin{proposition}\label{T^cc closed BM-dist}
	Let $ X $ be  a  Banach  space and $r\in [0,2]$. If for every $\delta>0$ there exists a Banach space $Y$ which is isomorphic to $X$ such that $d(X,Y) < 1+\delta$ with $\mathcal{T}^s(Y)=r$, then $\mathcal{T}^{s}(X)\geq r$.
\end{proposition}
\begin{proof}
 Let $S(B_X,\xast,\alpha)$ be a slice of $B_X$, $x\in S_X$, and $\varepsilon>0$. Let $\delta\in (0,\min\{\alpha,\varepsilon\})$. 
  
 Next find a Banach space $Y$ with $\mathcal{T}^s(Y) = r$
 such that $d(X,Y) < 1 + \delta$.
 Then there exist an isomorphism $L:X\rightarrow Y$ and such that $\|L\| = 1$ and $\|L^{-1}\| < 1+\delta$.
  
  Consider now
  \begin{equation*}
    \yast :=
    \frac{(L^{-1})^{\ast}\xast}{\|(L^{-1})^{\ast}\xast\|}\in S_{\Yast}
    \quad \text{ and }\quad
    y:=\frac{Lx}{\|Lx\|}\in S_Y.
  \end{equation*}
  Since $\mathcal{T}^s(Y)=r$, we can find a $v\in
  S(B_Y,y^*,\delta^2)$ such that
  $\|v-y\|\geq r-\varepsilon$. Denote by
  $u:=\frac{L^{-1}v}{(1+\delta)}$ and
  observe that $u\in S(B_X,\xast,\delta)\subset S(B_X, \xast,\alpha)$. 
  Our aim now is to
  show that $\|u-x\|\geq r-3\varepsilon$. Indeed,
  \begin{align*}
    r-\varepsilon
    &\leq
    \| v - \frac{Lx}{\|Lx\|}\|
    \leq
    \|L^{-1}v-\frac{x}{\|Lx\|}\|\\
    &\leq
    \|u - x\|
    +
    \|(1+\delta)u -u\|
    +
    \|x-\frac{x}{\|Lx\|}\|\\
    &\leq
    \|u - x\| + \delta + \delta\\
    &<
    \|u - x\| + 2\varepsilon.
  \end{align*}
  Hence, $\|u-x\|\geq \mathcal{T}^s(Y) - 3\varepsilon$
  and from the arbitrariness of
  $\varepsilon$, we have
  $\mathcal{T}^{s}(X) \geq \mathcal{T}^s(Y)$.
\end{proof}

An application of Proposition~\ref{T^cc closed BM-dist} together with Proposition~\ref{prop: equivalent Daugavet} immediately gives us that the class of Banach spaces with the Daugavet property is  closed  with  respect  to  the  Banach--Mazur  distance.

\begin{corollary}
Let $ X $ be  a  Banach  space. If for every $\delta>0$ there exists a Banach space $Y$ which is isomorphic to $X$ such that $d(X,Y) < 1+\delta$ and $Y$ has the Daugavet property, then $X$ also has the Daugavet property.
\end{corollary}

We will finish by connecting the current work with some
open questions related to a question of Ivakhno from 2006.

Recall that for a bounded set $C$ of a Banach space $X$ the \emph{radius of $C$} is defined as 
$$r(C):=\inf\{r>0: C\subseteq B(x,r)\mbox{ for some }x\in X\}.$$ A Banach space $X$ is said to have  the \emph{$r$-big slice property}
if every slice of $B_X$ is of radius one \cite{Ivakhno}. 
Observe that the $r$-big slice property of a
Banach space $X$ implies that $\mathcal{T}^s(X)\geq 1$.
Moreover, if every slice of $B_X$ has diameter two,
then $X$ has the $r$-big slice property.
Ivakhno asked if the converse is true \cite[p.~96]{Ivakhno}.

In view of Ivakhno's question, given a Banach space $X$, the following questions make sense:
  \begin{enumerate}
  \item \label{ques: rBSPa}
    If $\mathcal{T}^s(X) \ge 1$, then does every slice of $B_X$ have
    diameter two?
  \item \label{ques: weakopen}
    If $\mathcal{T}(X) \ge 1$,
    then does every nonempty relatively weakly open subset of $B_X$
    have diameter two?
  \item \label{ques: ccslices} If $\mathcal T^{cc}(X)\geq 1$, then
    does every convex combination of slices of $B_X$ have diameter
    two?
  \end{enumerate}

A negative answer to \ref{ques: ccslices} easily follows from our
results, as the following remark shows.

\begin{remark}
  Let $1<p<\infty$ and $Y$ be a Banach
  space with the Daugavet property, and take $X:=Y\oplus_p Y$, then
  $\mathcal{T}^{cc}(X) = 2^{1/p}>1$ (see Theorem~\ref{thm: T p}), but
  for every $\varepsilon>0$ there is a nonempty convex combination of
  slices with diameter less than $2^{1/p}+\varepsilon$
  \cite[Theorem~2.8]{hln}.
\end{remark}

We end this paper by proving that the answer
to Ivakhno's question (and henceforth, the answer to \ref{ques: rBSPa})
is negative. Indeed, we have the following result.

\begin{theorem}\label{theo:JTinfty}
  There exists a Banach space $X$ with the $r$-big slice property (hence $\mathcal T^s(X)\geq 1$) and there exists $\varphi\in S_{X^*}$ so that
  \begin{equation*}
    \inf\limits_{\alpha>0} \diam(S(B_X,\varphi,\alpha))\leq \sqrt{2}.
  \end{equation*}
\end{theorem}

In order to prove it let us introduce a bit of notation.
Let us define
\begin{equation*}
  T := \{(\alpha_1,\ldots, \alpha_k)\ \colon
  \ k\in\mathbb N, \alpha_1,\ldots, \alpha_k \in \mathbb{N}\}
  \cup \{\emptyset\}.
\end{equation*}
Given $(\alpha_1,\ldots, \alpha_k),(\beta_1,\ldots, \beta_p)\in T\setminus \{\emptyset\}$ we say that
\begin{equation*}
  (\alpha_1,\ldots, \alpha_k)
  \leq
  (\beta_1,\ldots, \beta_p)
  \Leftrightarrow
  \begin{cases}
    \vert (\alpha_1,\ldots, \alpha_k)\vert
    \leq \vert(\beta_1,\ldots, \beta_p)\vert \\
    \alpha_i = \beta_i\ \mbox{for all}\ 1 \leq i\leq k,
  \end{cases}
\end{equation*}
where $\vert (\alpha_1,\ldots, \alpha_n)\vert :=n$
and $\vert \emptyset\vert :=0$, and we declare $\emptyset:=\min T$.
This binary relation defines a partial order on $T$.

A \emph{segment} in $T$ is a totally ordered and
finite subset $S\subseteq T$.

Given $x:T\longrightarrow \mathbb R$, let us consider
\begin{equation*}
  \Vert x\Vert
  =
  \sup\left( \sum_{i=1}^n
    \left( \sum_{t\in S_i}x(t) \right)^2
  \right)^\frac{1}{2},
\end{equation*}
where the sup is taken over all families $\{S_1,\ldots,S_n\}$
of disjoint segments of $T$.

Now $JT_\infty$ is defined as the completion of  the space of
finitely nonzero functions defined on $T$ (i.e. functions $x:T\longrightarrow \mathbb R$ such that $\{t\in T\ |\ x(t)\neq 0\}$ is finite) for the above norm. Given
$\alpha\in T$ let us define
\begin{equation*}
  e_\alpha(\beta)
  :=
  \begin{cases}
    1 & \mbox{if } \beta=\alpha,\\
    0 & \mbox{otherwise.}
  \end{cases}
\end{equation*}
Then it is known that $\{e_\alpha\}_{\alpha\in T}$ is a Schauder basis
for $JT_\infty$ and that $JT_\infty$ is a dual space.
We denote by $\{e_\alpha^*\}_{\alpha\in T}$ the biorthogonal sequence
of $\{e_\alpha\}_{\alpha\in T}$ .
Then
$B_\infty := \overline{\linspan}\{e_\alpha^*\ /\ \alpha\in T\}$,
where the closure is taken in $JT_\infty^*$, is a complete
predual of $JT_\infty$.

The space $JT_\infty$ was  introduced in \cite{goma},
where it is proved that $B_{\infty}$   fails the
Radon--Nikod\'{y}m property.
Furthermore, every infinite-dimensional subspace
of $JT_\infty$ contains an isomorphic copy of $\ell_2$ and so
$JT_\infty$ does not contain isomorphic copies of $\ell_1$.

Let us start with the following lemma about weakly null sequences in
$JT_\infty$.

\begin{lemma}\label{lemma:weaknullJTinf}
  Let $\{t_n:n\in\mathbb N\}$ be the set of successors of a given
  element $t\in T$. Then $\{e_{t_n}\}_n$ is weakly null.
\end{lemma}

\begin{proof}
  Let $x:=\sum_{j=1}^m \alpha_j e_{t_j}$, where $\alpha_j\in \R$,
  and let us prove that
  $\Vert x\Vert\leq \left( \sum_{j=1}^m
    \alpha_j^2\right)^\frac{1}{2}$.
  To this end, pick a family of disjoint segments
  $S_1,\ldots, S_k$.
  Since $\{t_n\}$ are incomparable then,
  for every $i\in\{1,\ldots, k\}$ then
  $S_i\cap \{t_1,\ldots, t_m\}$ has, at most, one element.
  Define $A$ the set of those $i$ such that
  $ S_i\cap \{t_1,\ldots, t_m\}=\{t_{k_i}\}$.
  Notice that, since the $S_i$ are disjoint, then
  $t_{k_i}\neq t_{k_j}$ if $i\neq j$ with $i,j\in A$.
  Now
  \begin{equation*}
    \sum_{i=1}^m \left(
      \left(\sum_{t\in S_i}x(t)\right)^2
    \right)^\frac{1}{2}
    =
    \left(\sum_{i\in A} \alpha_{t_{k_i}}^2\right)^\frac{1}{2}
    \leq \left(\sum_{i=1}^m \alpha_{i}^2\right)^\frac{1}{2}.
  \end{equation*}
  Taking the supremum on the family of disjoint segments, and taking into
  account the definition of the norm of $JT_\infty$, we get
  \begin{equation*}
    \Vert x\Vert\leq \left(\sum_{i=1}^m \alpha_{i}^2\right)^\frac{1}{2}.
  \end{equation*}
  The previous estimate implies, by the arbitrariness of
  $m\in\mathbb N$ and $\alpha_1,\ldots,\alpha_n\in\mathbb R$,
  that the linear operator $\Phi:\ell_2\longrightarrow JT_\infty$
  given by
  \begin{equation*}
    \Phi(e_n)=e_{t_n}
  \end{equation*}
  is continuous.
  The $w-w$ continuity of $\Phi$ and
  the fact that $\{e_n\}\rightarrow^{w}0$
  in $\ell_2$ concludes the lemma.
\end{proof}

Now we are ready to prove Theorem \ref{theo:JTinfty}.

\begin{proof}[Proof of Theorem \ref{theo:JTinfty}]
  Let $X=B_\infty$,
  the predual of $JT_\infty$ described above.
  The existence of $\varphi\in S_{X^*}$ satisfying our
  requirements follows from \cite[Theorem 2.2]{blrjames}.
  For the remaining part, let us even prove that, given a $w^*$-slice $S$ of $B_{X^{**}}=B_{JT_\infty^*}$, we get that
  \begin{equation*}
   r(S)\geq 1.
  \end{equation*}
   To this end, pick a $w^*$-slice $S:=S(B_{JT_\infty^*},x,\alpha)$,
  for a suitable finitely-supported function
  $x:T\longrightarrow \mathbb R$ of norm one.
  Pick $x^*\in JT_\infty^*$ and $\varepsilon>0$,
  and let us find an element $y^*\in S$ with
  $\Vert x^*-y^*\Vert\geq 1-\varepsilon$.
  To this end, by \cite[Lemma 2.1]{blrjames}
  we can find an element $g\in S$ of the form
  $g:=\sum_{i=1}^n \lambda_i f_{S_i}$, where
  $\lambda_1,\ldots, \lambda_n \in\mathbb R^+$
  with $\sum_{i=1}^n \lambda_i^2=1$ and $S_1,\ldots, S_n$
  is a family of disjoint segments.

  Pick $i\in\{1,\ldots, n\}$.
  Since the set of offspring of a given element is infinite
  and the fact that the support of $x$ is finite,
  we can assume (by adding elements which do not
  belong to $\supp(x)$ into $S_i$ keeping the disjointness
  condition on the segments $S_1,\ldots, S_n$) that,
  if $t_i$ is the maximum element of $S_i$, then
  for every $z\geq t_i$ one has $z\notin \supp(x)$
  and that $\{t_1,\ldots, t_n\}$ are at the same level.
  For every $i\in\{1,\ldots, n\}$ fix $\{t_n^i\}$ the
  set of offspring of $t_i$.
  By Lemma \ref{lemma:weaknullJTinf}, it follows that $\{e_{t_n^i}\}$
  is weakly null, which means that
  $x^*(e_{t_n^i})\rightarrow 0$.
  Consequently, we can find $n$ large enough so that
  $x^*(e_{t_n^i}) < \frac{\varepsilon}{\lambda}$, for
  $\lambda := \min_{1\leq i\leq n} \lambda_i$, holds for every
  $i\in\{1,\ldots, n\}$. Define $R_i := S_i\cup \{t_n^i\}$ and notice
  that $y^*:=\sum_{i=1}^n \lambda_i f_{R_i}$ is a norm-one element
  (because $\{R_1,\ldots, R_n\}$ is still a family of disjoint
  segments) and that $y^*\in S$ (indeed, notice that $y^*(x)=g(x)$
  because $t_n^i\notin \supp(x)$ for every $i$). Define
  $z:T\longrightarrow \mathbb R$ by $z=\sum_{i=1}^n \lambda_i
  e_{t_n^i}$. Notice that $y^*(z)=\sum_{i=1}^n \lambda_i^2=1$ by
  assumptions. Moreover, similar estimates to the ones of
  Lemma~\ref{lemma:weaknullJTinf} prove that $\Vert z\Vert\leq 1$
  in $JT_\infty$. Moreover
  \begin{equation*}
    x^*(z)=\sum_{i=1}^n \lambda_i x^*(e_{t_n^i})<\varepsilon.
  \end{equation*}
  So
  \begin{equation*}
    \Vert y^*-x^*\Vert\geq (y^*-x^*)(z)>1-\varepsilon,
  \end{equation*}
  as desired.
\end{proof}

\begin{remark}
Notice that $\mathcal T(B_\infty)=0$ since the unit ball of $B_\infty$ contains nonempty relatively weakly open subsets of arbitrarily small diameter (in fact, $B_\infty$ has the \textit{convex point of continuity property} \cite[Theorem 2.2]{gms}). Consequently, question \ref{ques: weakopen} remains open.
\end{remark}

\bibliographystyle{plain}

\end{document}